\pdfoutput=1
\documentclass[11pt]{article}

\usepackage[utf8]{inputenc} 
\usepackage[T1]{fontenc}    
\usepackage{hyperref}       
\usepackage{url}            
\usepackage{booktabs}       
\usepackage{amsfonts}       
\usepackage{nicefrac}       
\usepackage{microtype}      
\usepackage{wrapfig}
\usepackage{bm}
\usepackage{xcolor}
\usepackage{enumitem}
\usepackage[lined,boxed]{algorithm2e}
\usepackage{caption}
\usepackage{subcaption}
\usepackage{tabularx}
\usepackage{paralist}
\usepackage{graphicx}
\usepackage{amsthm}
\usepackage[notheorems]{dgleich-math}

\usepackage{ushort}
\newcommand{\cel}[1]{\ushort{#1}}

\newcommand{\celm}[1]{\cel{\mat{#1}}}

\newcommand{\cP}{\ensuremath{\cel{P}}}

\providecommand{\cmP}{\ensuremath{\celm{P}}}

\usepackage{floatrow}
\newfloatcommand{capbtabbox}{table}[][\FBwidth]

\hypersetup{colorlinks,citecolor=blue}

\usepackage[paper=letterpaper, top=1.0in, bottom=1.0in, left=1.0in, right=1.0in]{geometry}

\newtheorem{theorem}{Theorem}[section]
\newtheorem{lemma}[theorem]{Lemma}

\newtheorem{corollary}[theorem]{Corollary}

\newcommand{\given}

\newcolumntype{Y}{>{\raggedright\arraybackslash}X}
\newcolumntype{W}{>{\raggedleft\arraybackslash}X}
\newcolumntype{Z}{>{\centering\arraybackslash}X}

\graphicspath{{figure/}}

\begin{document}
\title{Multi-way Monte Carlo Method for Linear Systems}

\author {Tao Wu\\
Purdue University\\
wu577@purdue.edu
\and
David F. Gleich\\ 
Purdue University\\
dgleich@purdue.edu
}

\date{}


\maketitle
\thispagestyle{empty}

\begin{abstract}
We study the Monte Carlo method for solving a linear system of the form $\vx = \mH \vx + \vb$.
A sufficient condition for the method to work
is $\normof{\mH} < 1$, which greatly limits the usability of this method.
We improve this condition by proposing a new multi-way Markov random walk,
which is a generalization of the standard Markov random walk. Under our new framework we prove that
the necessary and sufficient condition for our method to work is the spectral radius $\rho(\mH^{+}) < 1$, which is a 
weaker requirement than $\normof{\mH} < 1$. In addition to solving more
problems, our new method
can work faster than the standard algorithm. In numerical experiments on
both synthetic and real world matrices, we demonstrate the effectiveness of our
new method.
\end{abstract}
\newpage
\setcounter{page}{1}

\section{Introduction}
\label{sec_intro}
The Monte Carlo method~\cite{wasow1952note} for solving a linear system uses a random walk to approximate the solution. This method has several advantages over traditional deterministic algorithms (e.g., Gaussian elimination and iterative methods) due to its unique characteristics. First, the Monte Carlo method can be highly effective when only modest accuracy is required, as is common for many problems on data such as PageRank computations~\cite{avrachenkov2007monte}. Second, it is well-known that Monte Carlo algorithms
are highly parallelizable~\cite{lebeau1999parallel,dietrich1996scalar}, thus they are ideal for modern paralleled computers or clusters. Third, Monte Carlo methods can identify only on a single component or a linear form of the solution, which is often all that is required in many applications~\cite{wang2008monte}. Last but not least, Monte Carlo methods have advantages when dealing with large linear systems ~\cite{halton1994sequential,dimov1998new} because they do not always require a full solution vector. 

\subsection{The standard Monte Carlo method}
Consider the following linear system:
\[
\vx = \mH \vx + \vb
\]
where $\mH \in \mathbb{R}^{n\times n}$ and $\vx, \vb \in \mathbb{R}^{n}$ and where our goal is to evaluate the functional $\langle \vh , \vx \rangle = \sum_{i=1}^{n} h_i x_i$. We could then use this primitive to compute the solution by evaluating the functional for each standard basis vector to get each single component of the solution. 

It is known that if the spectral radius $\rho(\mH) < 1$, then the Neumann Series
$\sum_{\ell=0}^{\infty} \mH^{\ell}\vb$ will converge to the solution vector $\vx$.
The Monte Carlo method uses this observation to create a Markov random walk $X_t$ on the state space $S = \{1,2,\cdots, n\}$ with initial probability
$\text{Pr}(X_0 = i) = p_i$
and transition probability $\text{Pr}(X_{\ell+1} = j \mid X_{\ell} = i) = P_{i,j}$, s.t.$h_i\neq 0 \Rightarrow p_i \neq 0$
and $H_{i,j}\neq 0 \Rightarrow P_{i,j}\neq 0$.

Let $\nu$ be a realization of the random walk: $k_0 \rightarrow k_1 \rightarrow k_2 \rightarrow \cdots \rightarrow k_{\ell} \rightarrow \cdots$.
A walk related weight and random variable can be calculated as
\begin{equation*}
\begin{aligned}
W_{\ell} &= \frac{h_{k_0}H_{k_0,k_1}H_{k_1,k_2}\cdots H_{k_{\ell-1},k_{\ell}}}{p_{k_0}P_{k_0,k_1}P_{k_1,k_2}\cdots P_{k_{\ell-1},k_{\ell}}} \quad \text{for }\ell=0,1,2,\cdots & \quad \text{ and } \quad & 
X(\nu) = \sum_{\ell = 0}^{\infty} W_{\ell}b_{k_{\ell}}.
\end{aligned}
\end{equation*}
Then it can be shown (for instance~\cite{dimov1998new}) that $\E[X] = \langle \vh,\vx \rangle$, and more specifically $\E[W_{\ell}f_{k_{\ell}}] = \langle \vh, \mH^{\ell}\vf \rangle$.

However the random walk model does not guarantee convergence ~\cite{ji2013convergence,wasow1952note}. According to the law of large numbers, a necessary and sufficient condition to estimate $\E[X]$ using the empirical mean value of $X$ is $\Var[X]<\infty$. Empirical studies~\cite{ji2013convergence,benzi13analysis} show that it is easy to have $\Var[X] = \infty$ 
even when the Neumann Series converges (i.e., $\rho(\mH)<1$). 

\subsection{Our Contributions}
In order to apply the Monte Carlo method, existing work~\cite{dimov2015new,srinivasan2010monte,wang2008monte} assumes $\normof{\mH} < 1$ (for the infinity norm $\normof{\mH} = \max_i \sum_j |H_{i,j}|$), which suffices to show $\Var[X] < \infty$, but which is a stronger condition than $\rho(\mH) < 1$. Although it is possible $\Var[X]<\infty$ when $\normof{\mH} \geq 1$, there is no easy way to check. To tackle this problem, we propose a multi-way Markov random walk which uses multiple transition matrices. At each step of the random walk, the transition matrix is constructed in a way akin to the \textit{Monte Carlo Almost Optimal (MAO)} framework~\cite{dimov2015new,ji2013convergence}. 
We prove that under this type of random walk, the new method always converges when $\rho(\mH^{+}) < 1$, where $\mH^{+}$ is the nonnegative matrix as $H^{+}_{ij} = |H_{ij}|$.
In addition, our new framework has the tendency to get the result faster than the standard method. One downside to our approach is that is cannot be implemented in a purely local fashion akin to the standard Monte Carlo method as it requires global work to build the multi-way walk. 

\subsection{Related Work}
Research on Monte Carlo Methods for linear systems can be divided into two classes: direct methods and hybrid methods. Direct methods
study the various techniques of using the Monte Carlo solvers themselves, for instances non-diagonal
splitting~\cite{srinivasan2010monte} and relaxation parameters~\cite{dimov1998new}.
Hybrid methods~\cite{halton1994sequential,evans2014monte,alexandrov2005parallel} use Monte Carlo as a black box combined with iterative techniques.
Examples of these works are Sequential Monte Carlo method~\cite{halton1994sequential} and 
synthetic-acceleration method~\cite{evans2014monte}. Also there are a variety of studies of the parallel implementation~\cite{dimov2001parallel,slattery2013parallel,alexandrov2005parallel}, real world application~\cite{wang2008monte,avrachenkov2007monte}, convergence analysis~\cite{ji2013convergence}, and spectral
analysis~\cite{slattery2015spectral}.

In this paper, we focus on the direct Monte Carlo procedure.  Our ideas can also be incorporated into the hybrid frameworks to better improve the performance. 

\section{Multi-way Markov Random Walk}
\label{sec_ranwalk}
In this section we generalize the idea of random walk for estimating the functional to using a hypermatrix of transitions to compute the estimate. Then we analyze the convergence of the simulations based on the variance of the relevant random variable. 

We use bold, upper-case letters such as $\mA$ to denote matrices,
and bold, lower-case letters such as $\vx$ to denote vectors.
Hypermatrices as in $\cmP$ are bold, underlined, upper-case letters.
We use letters with subscripts of indices to denote elements $x_i$ of a vector
 and $A_{i,j}$ of a matrix. For a mode-$3$ hypermatrix $\cmP$, its elements
 are denoted by $\cP^{(\ell)}_{i,j}$.

\subsection{Hypermatrix Transitions}

Instead of using a fixed transition matrix $\mP$ as in the classic Monte Carlo method in section~\ref{sec_intro}, we allow the random
walk to vary transition matrices with each step. An $m-$way random walk can be interpreted as walking via $m$ different transition matrices periodically in a round-robin way. 
Formally, we define a $m-$way Markov random walk $(\vp,\cmP)$ as $Z_t$: $k_0 \rightarrow k_1 \rightarrow k_2 \rightarrow \cdots \rightarrow k_i \rightarrow \cdots$, where the initial probability follows $\vp$, and the transition probability follows a hypermatrix $\cmP$:
\begin{equation}
\label{equ_transition}
\begin{aligned}
&\text{Pr}(k_0 = i) = p_i\\
&\text{Pr}(k_{\ell+1} = j | k_{\ell} = i) = \cP^{(\text{mod}(\ell, m)+1)}_{i,j}.
\end{aligned}
\end{equation}
Here $\text{mod}(\ell, m)$ denotes the remainder after dividing $\ell$ by $m$. For notation simplicity, we use $\cP^{(\ell)}_{i,j}$ to denote $\cP^{(\text{mod}(\ell-1, m)+1)}_{i,j}$ for $\ell=1,2,\cdots$.

\subsection{The Multi-way Monte Carlo Method}
Our goal is to compute the functional $\langle \vh, \vx \rangle$ where $\vx$ is the solution of linear system $\vx = \mH\vx + \vb$. Through the paper we have the basic assumption $\rho(\mH)<1$. We also exclude the corner cases where $\vh$ is a zero vector,
or $\mH$ has zero rows/columns.
If we construct the initial probability such that $h_i\neq 0 \Rightarrow p_i\neq 0$ and the transition hypermatrix $\cmP$ such that $H_{i,j}\neq 0 \Rightarrow \cP^{(\ell)}_{i,j}\neq 0$,
then we can define the related weights $W_{\ell}$ and the variable $Z$ in a similar way with section~\ref{sec_intro}, and formally:
\begin{equation}
\label{equ_weight}
\begin{aligned}
&W_{\ell} = \frac{h_{k_0}H_{k_0,k_1}H_{k_1,k_2}\cdots H_{k_{\ell-1},k_{\ell}}}{p_{k_0}\cP^{(1)}_{k_0,k_1}\cP^{(2)}_{k_1,k_2}\cdots \cP^{(\ell)}_{k_{\ell-1},k_{\ell}}}
\quad \text{for }\ell=0,1,2,\cdots \\
&Z = \sum_{\ell = 0}^{\infty} W_{\ell}b_{k_{\ell}}
\end{aligned}
\end{equation}
It is worth noting the above definition of multi-way Markov random walk is a
generalization of the standard Markov chain, which is the special case with $m=1$.

\begin{theorem}
\label{the_exp}
For the linear system $\vx = \mH\vx + \vb$, $Z$ defined from~\eqref{equ_weight} has the expected value $\E[Z] = \langle \vh,\vx \rangle$.
\end{theorem}

\begin{proof}
We first prove that $\E[W_{\ell}b_{k_\ell}] = \langle \vh, \mH^{\ell}\vb \rangle$ for all $\ell = 0,1,2,\cdots$. Then the convergence of the Neumann series will give us $\E[Z] = \langle \vh,\vx \rangle$.

We have $\E[W_0b_{k_0}] = \sum_{p_{k_0}\neq 0} \frac{h_{k_0}}{p_{k_0}} b_{k_0} p_{k_0} = \sum_{h_{k_0}\neq 0} h_{k_0} b_{k_0} = \langle \vh,\vb \rangle$. Similarly for the case of $\ell \geq 1$:
\begin{equation*}
\begin{aligned}
\E[W_{\ell}b_{k_\ell}] &= \sum_{p_{k_0}\neq 0}\sum_{\cP^{(1)}_{k_0,k_1}\neq 0}\cdots\sum_{\cP^{(\ell)}_{k_{\ell-1},k_{\ell}}\neq 0}
\frac{h_{k_0}H_{k_0,k_1}H_{k_1,k_2}\cdots H_{k_{\ell-1},k_{\ell}}}{p_{k_0}\cP^{(1)}_{k_0,k_1}\cP^{(2)}_{k_1,k_2}\cdots \cP^{(\ell)}_{k_{\ell-1},k_{\ell}}}
b_{k_{\ell}}
p_{k_0}\cP^{(1)}_{k_0,k_1}\cP^{(2)}_{k_1,k_2}\cdots \cP^{(\ell)}_{k_{\ell-1},k_{\ell}}\\
&= \sum_{h_{k_0}\neq 0}\sum_{H_{k_0,k_1}\neq 0}\cdots\sum_{H_{k_{\ell-1},k_{\ell}}\neq 0}
h_{k_0}H_{k_0,k_1}H_{k_1,k_2}\cdots H_{k_{\ell-1},k_{\ell}}
b_{k_{\ell}}\\
& = \sum_{k_0=1}^{n} \sum_{k_1=1}^{n} \cdots\sum_{k_{\ell}=1}^{n}
h_{k_0}H_{k_0,k_1}H_{k_1,k_2}\cdots H_{k_{\ell-1},k_{\ell}}
b_{k_{\ell}}
= \langle \vh, \mH^{\ell}\vb \rangle
\end{aligned}
\end{equation*}
So $\E[Z] = \sum_{\ell=0}^{\infty} \E[W_{\ell}b_{k_{\ell}}] = \langle \vh, \sum_{\ell=0}^{\infty} \mH^{\ell}\vb \rangle = \langle \vh,\vx \rangle$.
\end{proof} 


\subsection{Convergence Analysis}
In order to statistically estimate $\E[Z]$, we need to ensure $\Var[Z] < \infty$. The following theorem reveals the explicit form of $\Var[Z]$ determined by $\vh,\vb,\mH$ and the $m-$way random walk $(\vp,\cmP)$.

\begin{theorem}
\label{the_var}
 For the linear system $\vx = \mH\vx + \vb$, if $\mH$ and $\vb$ are nonnegative, $Z$ defined from~\eqref{equ_weight} has variance 
\begin{equation}
\label{equ_variance}
\Var[Z] = \langle \hat{\vh},\sum_{i=0}^{\infty}\tilde{\mH}^{i}\mG \Diag(\vb) (2\mH\vx + \vb) \rangle
-\langle \vh,\vx \rangle ^2
\end{equation}
where $\Diag(\vb)$ is a diagonal matrix with diagonal entries equal to $\vb$, and $\hat{\vh}, \tilde{\mH}, \mG$ are defined as:
\begin{equation*}
\begin{aligned}
&\hat{h}_i = \begin{cases}
h_i^2/p_i & \text{if } h_i\neq 0\\
0 & \text{if } h_i = 0
\end{cases}
\qquad
\hat{H}^{(\ell)}_{i,j} = \begin{cases}
H^{2}_{i,j}/\cP^{(\ell)}_{i,j} & \text{if } H_{i,j}\neq 0\\
0 & \text{if } H_{i,j} = 0
\end{cases}\\
&\tilde{\mH} = \hat{\mH}^{(1)}\hat{\mH}^{(2)}\cdots\hat{\mH}^{(m)}
\qquad
\mG = \mI + \hat{\mH}^{(1)} + \hat{\mH}^{(1)}\hat{\mH}^{(2)} + \cdots + \hat{\mH}^{(1)}\hat{\mH}^{(2)}\cdots\hat{\mH}^{(m-1)}
\end{aligned}
\end{equation*}
\end{theorem}

\begin{proof}
Since $\Var[Z] = \E[Z^2] - (\E[Z])^2 = \E[Z^2] - \langle \vh,\vx \rangle ^2$, we will focus on computing $\E[Z^2]$:
\begin{equation*}
\E[Z^2] = \E[\sum_{\ell=0}^{\infty}W_{\ell}^{2}b_{k_\ell}^{2} + 2\sum_{r>\ell} W_{\ell}W_{r}b_{k_{\ell}}b_{k_r}]
\end{equation*}
Since all the intermediate terms are nonnegative, by Tonelli's theorem we can analyze the sum in pieces.
We have $\E[W_{0}^2b_{k_0}^2] = \sum_{p_{k_0}\neq 0} \frac{h^{2}_{k_0}}{p^{2}_{k_0}} b^{2}_{k_0} p_{k_0} = \sum_{\hat{h}_{k_0}\neq 0} \hat{h}_{k_0} b^{2}_{k_0} = \langle \hat{\vh},\Diag(\vb)\vb \rangle$, and when $\ell \geq 1$,
\begin{equation}
\label{equ_wb}
\begin{aligned}
\E[W_{\ell}^2b_{k_{\ell}}^2] &= \hspace{-0.1cm} \sum_{p_{k_0}\neq 0}\sum_{\cP^{(1)}_{k_0,k_1}\neq 0}\cdots\sum_{\mathclap{\cP^{(\ell)}_{k_{\ell-1},k_{\ell}}\neq 0}}\,\,
\Big( \frac{h_{k_0}H_{k_0,k_1}H_{k_1,k_2}\cdots H_{k_{\ell-1},k_{\ell}}}{p_{k_0}\cP^{(1)}_{k_0,k_1}\cP^{(2)}_{k_1,k_2}\cdots \cP^{(\ell)}_{k_{\ell-1},k_{\ell}}} \Big)^2 
b^{2}_{k_{\ell}}
p_{k_0}\cP^{(1)}_{k_0,k_1}\cP^{(2)}_{k_1,k_2}\cdots \cP^{(\ell)}_{k_{\ell-1},k_{\ell}}\\
& = \hspace{-0.1cm} \sum_{\hat{h}_{k_0}\neq 0}\sum_{\hat{H}^{(1)}_{k_0,k_1}\neq 0}\cdots\sum_{\mathclap{\hat{H}^{(\ell)}_{k_{\ell-1},k_{\ell}}\neq 0}}\,
\hat{h}_{k_0}\hat{H}^{(1)}_{k_0,k_1}\hat{H}^{(2)}_{k_1,k_2}\cdots \hat{H}^{(\ell)}_{k_{\ell-1},k_{\ell}} b^{2}_{k_{\ell}}\\
& = \langle \hat{\vh}, \hat{\mH}^{(1)}\hat{\mH}^{(2)}\cdots\hat{\mH}^{(\ell)}\Diag(\vb)\vb \rangle
\end{aligned}
\end{equation}
Applying the above result from~\eqref{equ_wb}, we have
\begin{equation}
\label{equ_sumwb}
\begin{aligned}
E&[\sum_{\ell=0}^{\infty}W_{\ell}^{2}b_{k_\ell}^{2}]
=
\Big \langle \hat{\vh}, \big(\mI + \sum_{\ell=1}^{\infty}\hat{\mH}^{(1)}\hat{\mH}^{(2)}\cdots\hat{\mH}^{(\ell)}\big)\Diag(\vb)\vb \Big \rangle \\
& = 
\Big \langle \hat{\vh}, \big(\mG + \sum_{\ell=m}^{\infty}\hat{\mH}^{(1)}\hat{\mH}^{(2)}\cdots\hat{\mH}^{(\ell)}\big)\Diag(\vb)\vb \Big \rangle\\
& = 
\Big \langle \hat{\vh}, \big(\mG + \tilde{\mH}(\mI + \sum_{\ell=1}^{\infty}\hat{\mH}^{(1)}\hat{\mH}^{(2)}\cdots\hat{\mH}^{(\ell)})\big)\Diag(\vb)\vb \Big \rangle
= 
\Big \langle \hat{\vh}, \sum_{i=0}^{\infty}\tilde{\mH}^{i}\mG \Diag(\vb)\vb \Big \rangle
\end{aligned}
\end{equation}

Next we compute the second part of $\E[Z^2]$:
\begin{eqnarray}
\label{equ_wwbb}
\nonumber
& &\E[\sum_{r>\ell} W_{\ell}W_{r}b_{k_{\ell}}b_{k_r}] = \E[\sum_{\ell=0}^{\infty} W_{\ell} b_{k_{\ell}}( \sum_{r=\ell+1}^{\infty} W_{r}b_{k_r})]\\ \nonumber
& & = \hspace{-0.05cm} \sum_{\ell=0}^{\infty} 
\sum_{p_{k_0}\hspace{-0.05cm}\neq 0} \cdots\sum_{\mathclap{\cP^{(\ell)}_{k_{\ell-1},k_{\ell}}\hspace{-0.05cm} \neq 0}}\,
\Big(\frac{h_{k_0}H_{k_0,k_1}H_{k_1,k_2} \hspace{-0.1cm} \cdots H_{k_{\ell-1},k_{\ell}}}{p_{k_0}\cP^{(1)}_{k_0,k_1}\cP^{(2)}_{k_1,k_2} \hspace{-0.1cm} \cdots \cP^{(\ell)}_{k_{\ell-1},k_{\ell}}}\Big)^2
\hspace{-0.05cm} \Big(\hspace{-0.15cm} \sum_{r=\ell+1}^{\infty} \sum_{\cP^{(\ell)}_{k_{\ell},k_{\ell+1}} \hspace{-0.05cm} \neq 0} \cdots\sum_{\mathclap{\cP^{(\ell)}_{k_{r-1},k_{r}}\hspace{-0.05cm} \neq 0}}\hspace{0.35cm}
\frac{H_{k_{\ell},k_{\ell+1}} \hspace{-0.1cm} \cdots H_{k_{r-1},k_{r}}}{\cP^{(\ell+1)}_{k_{\ell},k_{\ell+1}} \hspace{-0.1cm} \cdots \cP^{(r)}_{k_{r-1},k_{r}}} \\ \nonumber
& &\ \quad \times 
p_{k_0}\cP^{(1)}_{k_0,k_1}\cP^{(2)}_{k_1,k_2} \cdots \cP^{(\ell)}_{k_{\ell-1},k_{\ell}}
\cP^{(\ell+1)}_{k_{\ell},k_{\ell+1}} \hspace{-0.1cm} \cdots \cP^{(r)}_{k_{r-1},k_{r}}
b_{k_\ell}b_{k_r} \Big),
\end{eqnarray}
(here, we have extracted all the prefix terms in $W_\ell$ and $W_r$ that are the same because $r > \ell$)
\begin{eqnarray}
\nonumber & & = \hspace{-0.05cm} \sum_{\ell=0}^{\infty} 
\sum_{\hat{h}_{k_0}\hspace{-0.05cm}\neq 0} \cdots\sum_{\mathclap{\hat{H}^{(\ell)}_{k_{\ell-1},k_{\ell}}\hspace{-0.05cm} \neq 0}}
\Big(\hat{h}_{k_0}\hat{H}^{(1)}_{k_0,k_1}\hat{H}^{(2)}_{k_1,k_2} \hspace{-0.1cm} \cdots \hat{H}^{(\ell)}_{k_{\ell-1},k_{\ell}} \Big)
b_{k_\ell}
\Big(\hspace{-0.15cm} \sum_{r=\ell+1}^{\infty} \sum_{H_{k_{\ell},k_{\ell+1}} \hspace{-0.05cm} \neq 0} \cdots\sum_{\mathclap{H_{k_{r-1},k_{r}} \hspace{-0.05cm} \neq 0}}
H_{k_{\ell},k_{\ell+1}} \hspace{-0.1cm} \cdots H_{k_{r-1},k_{r}} \Big)b_{k_r}\\ 
& &= \Big \langle \hat{\vh},
\sum_{\ell=0}^{\infty} (\hat{\mH}^{(1)}\cdots\hat{\mH}^{(\ell)})
\Diag(\vb) (\sum_{r=\ell+1}^{\infty} \mH^{r-\ell} \vb) \Big \rangle\\ \nonumber
& &= \Big \langle \hat{\vh},
\sum_{\ell=0}^{\infty} (\hat{\mH}^{(1)}\cdots\hat{\mH}^{(\ell)})
\Diag(\vb) \mH\vx \Big\rangle
= \Big \langle \hat{\vh}, \sum_{i=0}^{\infty}\tilde{\mH}^{i}\mG \Diag(\vb)\mH\vx \Big \rangle.
\end{eqnarray}
For these final steps, we used the Neumann series to move to $\mH\vx$ and then used the periodicity to rewrite the expressions in terms of $\mG$. 
Now, combining the results from~\eqref{equ_sumwb} and~\eqref{equ_wwbb} we have:
\begin{equation*}
\begin{aligned}
\Var[Z]
& = \Big \langle \hat{\vh}, \sum_{i=0}^{\infty}\tilde{\mH}^{i}\mG \Diag(\vb)\vb \Big \rangle
+ 2 \Big \langle \hat{\vh}, \sum_{i=0}^{\infty}\tilde{\mH}^{i}\mG \Diag(\vb)\mH\vx \Big \rangle
-\langle \vh,\vx \rangle ^2\\
& =
\langle \hat{\vh},\sum_{i=0}^{\infty}\tilde{\mH}^{i}\mG \Diag(\vb) (2\mH\vx + \vb) \rangle
-\langle \vh,\vx \rangle ^2 \qedhere
\end{aligned}
\end{equation*}
\end{proof}

For the general cases of $\mH,\vb$ without the assumption of nonnegativity, if $\rho(\tilde{\mH}) < 1$, the above conclusion (i.e., equation \eqref{equ_variance}) still holds according to Fubini's Theorem.

Combining both of these results, the following corollary is straightforward from the conclusion of Theorem~\ref{the_var}.

\begin{corollary}
\label{cor_1}
For the linear system $\vx = \mH\vx + \vb$, if the spectral radius $\rho(\tilde{\mH}) < 1$, then $\Var[Z] = \langle \hat{\vh},(\mI - \tilde{\mH})^{-1}\mG \Diag(\vb) (2\mH\vx + \vb) \rangle
-\langle \vh,\vx \rangle ^2 < \infty$
\end{corollary}


The above analysis of $\Var[Z]$ shows that with the condition $\rho(\tilde{\mH}) < 1$, and by the law of large numbers we can estimate the value of $\langle \vh,\vx \rangle$ from the variable $Z$.
For the cases when $\rho(\tilde{\mH}) \geq 1$, the following corollary shows that it is possible to have $\Var[Z] = \infty$. The essence of the idea and proof is just that we can construct a vector to touch the dominant eigenvector with eigenvalue $\ge 1$.

\begin{corollary}
\label{cor_2}
Under the same assumptions with Theorem~\ref{the_var}, if the spectral radius $\rho(\tilde{\mH}) \geq 1$, and if $\mG$ is full-rank, then there always exists some $\vb,\vh\in \mathbb{R}^n$ such that $\Var[Z] = \infty$. (Note that for the standard Monte Carlo method (i.e., $m=1$), since
$\mG = \mI$, the method diverges for certain $\vb,\vh$.)
\end{corollary}

\begin{proof}
Let $\mJ$ denote the Jordan canonical form for matrix $\hat{\mH}$ s.t. $\hat{\mH} = \mP \mJ \mP^{-1}$, where $\mP = [\vp_1,\vp_2,\cdots,\vp_n]$ and $\vp_i$ for $i=1,2,\cdots,n$ are the generalized eigenvectors. The diagonal entries of $\mJ$ are eigenvalues of $\hat{\mH}$, and $\mJ$ is composed with Jordan blocks:
\begin{equation}
\mJ = \pmat{\mJ_1 & & &\\
 & \mJ_2 & &\\
 & & \ddots &\\
 & & & \mJ_p}
 \text{\quad where \quad }
 \mJ_{i} = \pmat{
\lambda_i & 1 & & \\
& \lambda_i & \ddots & \\
& & \ddots & 1 &\\
& & & \lambda_i
 }
 \text{\quad for }i=1,2,\cdots,p.
\end{equation}

The power of $\mJ$ has the form: $\mJ^j = \Diag(\mJ_1^j, \mJ_2^j, \cdots, \mJ_p^j )$, where each individual block $\mJ_i^j$ with size $s$ is:
\begin{equation}
\mJ_i^j = \pmat{
	\lambda_i^j & \spmat{j\\1}\lambda_i^{j-1} & \cdots & \spmat{j\\s-1}\lambda_i^{j-s-1}\\
	0 & \lambda_i^j & \cdots & \spmat{j\\s-2}\lambda_i^{j-s} \\
	0 & \ddots & \ddots & \\
	0 & \cdots & \cdots & \lambda_i^j

}
\text{\quad for } j>s.
\end{equation}
So the upper right element $\spmat{j\\s-1}\lambda_i^{j-s-1}$ has the largest asymptotic value as $j\rightarrow \infty$.
Without a loss of generality, we can assume that $\mJ_1,\mJ_2,\cdots,\mJ_p$ are sorted in the decending order of eigenvalues, and for the case of the equal eigenvalues, they are sorted in the decending order of block sizes. So let $\mJ_1,\cdots,\mJ_k$ be the blocks with largestest eigenvalues (i.e., $\lambda_1 = \cdots = \lambda_k \geq \lambda_{k+1} \geq \cdots \geq \lambda_p$) and they have the same size $s$.

Denote $\vy = \mP^{-1}\mG \Diag(\vb)(2\mH\vx+\vb)$, and $\vz^{(i)} = \mJ^{i}\vy$ for $i=0,1,2,\cdots,n$. Given $y_s, y_{2s} \cdots, y_{ks} \neq 0$, we have $z^{(i)}_s = \lambda^{i}_{1}y_s (1 + o(1))$,
$z^{(i)}_{2s} = \lambda^{i}_{1}y_{2s} (1 + o(1))$,
$\cdots$
$z^{(i)}_{ks} = \lambda^{i}_{1}y_{ks} (1 + o(1))$,
 and $z^{(i)}_r/z^{(i)}_1 = o(1)$ for $r \neq s,2s,\cdots,ks$ as $i\rightarrow \infty$.
If we select $\vh$ s.t. $\langle \vh,y_s\vp_s+y_{2s}\vp_{2s}+\cdots+y_{ks}\vp_{ks} \rangle \neq 0$,
then we have:
\begin{equation*}
\begin{aligned}
& \langle \vh, \sum_{i=0}^{\infty}\tilde{\mH}^{i} \mG \Diag(\vb)(2\mH\vx+\vb) \rangle
 =\langle \vh,  \sum_{i=0}^{\infty} \mP \vz^{(i)} \rangle\\
 = & \sum_{i=0}^{\infty} z^{(i)}_1 \langle \vh,\vp_1 \rangle + \sum_{i=0}^{\infty} z^{(i)}_2 \langle \vh,\vp_2 \rangle + \cdots + \sum_{i=0}^{\infty} z^{(i)}_n \langle \vh,\vp_n \rangle\\
 =  &\big( \sum_{i=0}^{\infty} \lambda^{i}_1  \langle \vh,y_s\vp_s+y_{2s}\vp_{2s}+\cdots+y_{ks}\vp_{ks} \rangle \big) \big( 1 + o(1)\big) = \infty
\end{aligned}
\end{equation*}
Since $\vp_{s},\vp_{2s},\cdots,\vp_{ks}$ are linear independent, there always exists a vector $\vh$, s.t. $\langle \vh,y_s\vp_s+y_{2s}\vp_{2s}+\cdots+y_{ks}\vp_{ks} \rangle \neq 0$, given $y_{s},y_{2s},\cdots,y_{ks}$ are not all zero. Next we prove that there always exists a vector $\vx$ s.t. $y_s \neq 0$. Let $\vu^T$ denote the $s-$th row of $\mP^{-1}\mG$, then $y_{s}$ can be calculated as:
\begin{equation}
\label{equ_app}
\begin{aligned}
&y_s = \vu^T \Diag(\vb)(2\mH\vx + \vb) \\
 = &\vu^T \Diag(\vx - \mH\vx)(\vx + \mH\vx ) = \vu^T \Diag(\vx)\vx - \vu^T \Diag(\mH\vx)\mH\vx
\end{aligned}
\end{equation}
So $y_{s}$ is a polynomial of $\vx$ with coefficients coming from $\vu^T$ and $\mH$. If $y_{s} = 0$
for all $\vx\in \mathbb{R}^{n}$, then all the coefficients from equation~\ref{equ_app} are zero. Denote $h_i$ for $i=1,2,\cdots,n$ are the columns of matrix $\mH$, then the coefficients of terms $x_1^2,x_1x_2,x_1x_3,\cdots,x_1x_n$ equaling zero gives us:
\begin{equation*}
\begin{cases}
u_1 &- \sum_{i=1}^n u_i H_{i,1}H_{i,1} = 0\\
&- \sum_{i=1}^n u_i H_{i,1}H_{i,2} = 0\\
&- \sum_{i=1}^n u_i H_{i,1}H_{i,3} = 0\\
& \ \ \quad \vdots \qquad \vdots \qquad \qquad \quad \vdots\\
&- \sum_{i=1}^n u_i H_{i,1}H_{i,n} = 0
\end{cases}
\quad \Longrightarrow \quad
\mH^T \Diag(\vu) \vh_1 = (u_1,0,0,\cdots,0)^T
\end{equation*}
Similarly by setting the coefficients of terms $x_ix_1,x_ix_2,\cdots,x_ix_n$ to zero, we have equation: 
\begin{equation*}
\mH^T \Diag(\vu) \vh_i = (0,\cdots,0,u_i,0,\cdots,0)^T.
\end{equation*}
And combining all together will get us $\mH^T \Diag(\vu) \mH = \Diag(\vu)$. Since $\vu^T$ is the first row of a full-rank matrix, it cannot be a vector of all zeros, and the spectral radius
$\rho(\Diag(\vu)) \leq \rho(\mH^T) \rho(\Diag(\vu)) \rho(\mH) < \rho(\Diag(\vu))$ gives us the contradiction. So $y_{s}$ cannot always be zero.
\end{proof}

In this section we have seen that $\E[Z] = \langle \vh,\vx \rangle$, which provides us the potential to estimate the value of $\langle \vh,\vx \rangle$ by simulating the value of $Z$. However whether it is feasible to apply Monte Carlo simulation depends on $\tilde{\mH}$. If $\rho(\tilde{\mH}) < 1$, then $\Var[Z]<\infty$, so the simulation is guaranteed to converge. And if $\rho(\tilde{\mH}) \geq 1$, the simulation tends to fail.

\section{Multi-way Monte Carlo Method}
\label{sec_multiway}
In this section we discuss the two aspects of applying Monte Carlo method
based on the multi-way Markov random walk introduced in Section~\ref{sec_ranwalk}. First, we detail the 
construction of the transition hypermatrix $\cmP$. Second, we give the error analysis regarding the
truncation of the random walk, as well as the probable error.

\subsection{Transition Hypermatrix}
In section~\ref{sec_ranwalk} Corollary~\ref{cor_1} and~\ref{cor_2} indicate that the spectral radius of matrix
$\tilde{\mH}$ is crucial to the variance $\Var[Z]$. The matrix $\tilde{\mH}$ is
determined by the transition hypermatrix $\cmP$. Since it is usually computationally
inefficient to directly compute the spectral radius of a matrix, the common
practice is to find an upper-bound of $\rho(\tilde{\mH})$. The spectral radius of a matrix is bounded by any sub-multiplicative matrix norm. As before, we use the infinity norm $\normof{\cdot} \stackrel{\mathclap{\small\normalfont\mbox{def}}}{=} \normof[\infty]{\cdot}$ in this paper.

We first consider the case for the standard Markov random walk (i.e., $m=1$), where $\tilde{H}_{i,j} = H^{2}_{i,j}/\cP^{(1)}_{i,j}$. The following lemma~\cite{ji2013convergence} provides insight on how to assign the probability in terms of minimizing the norm.
\begin{lemma}
\label{lem_bound}
Let $\vh = (h_1,h_2,\cdots,h_n)^T$ be a vector where at least one of its elements is non-zero:
$h	_k\neq 0$ for some $k\in \{1,2,\cdots,n\}$. Let $\vp = (p_1,p_2,\cdots,p_n)^T$ be a probability
distribution vector. Then $\sum_{i=1}^{n} h^{2}_{i}/p_i \geq \Big(\sum_{i=1}^{n}|h_i|\Big)^2$, and
the lower-bound is attained when $p_i = |h_i|/\sum_{r=1}^{n}|h_r|$.
\end{lemma}
According to Lemma~\ref{lem_bound}, the infinity norm:
\begin{equation*}
\normof{\tilde{\mH}} = \max_{1\leq i\leq n} \sum_{j=1}^{n}|\tilde{H}_{i,j}| 
=  \max_{1\leq i\leq n} \sum_{j=1}^{n}\frac{H^{2}_{i,j}}{\cP^{(1)}_{i,j}} \geq 
\max_{1\leq i\leq n} \Big(\sum_{j=1}^{n}|H_{i,j}| \Big)^2 = (\normof{\mH})^2.
\end{equation*}
When $\cP^{(1)}_{i,j} = \frac{|H_{i,j}|}{\sum_{k=1}^{n}|H_{i,k}|}$ for all $i,j = 1,2,\cdots,n$, 
the above lower-bound $(\normof{\mH})^2$ is reached, making this choice in some sense optimal. However, for a variety of problems, this choice is unlikely to result in a method that will have $\rho(\tilde{\mH}) < 1$. 
For a linear system $\mA\vx = \vb$,
we can rewrite it into $\vx = \mH\vx + \vb$ as $\mH = \mI - \mA$. It is common to have $\rho(\mH) $ be very
close to $1$ even with the help of preconditioners~\cite{benzi13analysis}. Since the infinity norm is generally
a loose upper-bound for the spectral radius, $\normof{\mH} > 1$ is likely~\cite{benzi13analysis}. This inability of
upper-bounding the spectral radius $\rho(\tilde{\mH})$ for the standard Markov random walk encourages us to explore the multi-way generality.

We describe the method for computing $\cmP$ for an $m-$way Markov random walk in Algorithm~\ref{alg_1}, then we prove in Theorem~\ref{the_mini} that it minimizes $\normof{\tilde{\mH}}$.

\vspace{4pt}

\begin{algorithm}[H]
\label{alg_1}
 \SetAlgoLined
 \KwData{matrix \mH}
 \KwResult{transition hypermatrix $\cmP$}
 initialization $\omega_i = 1$ for $i=1,2,\cdots, n$\;
 \For{$k = m:1$}{
 $\eta_i = \sum_{\ell=1}^{n}\omega_{\ell}|H_{i,\ell}|$ \quad for $i=1,2,\cdots,n$\;
  $\cP^{(k)}_{i,j} = \omega_{j}|H_{i,j}|/\eta_i$
  \quad for $i,j = 1,2,\cdots, n$\;
  $\omega_i = \eta_i$ \quad for $i=1,2,\cdots,n$
 }
 \caption{Compute transition hypermatrix}
\end{algorithm}

\vspace{4pt}
It is worth noting that Algorithm~\ref{alg_1} only takes linear time in  the number
of non-zeros in the matrix $\mH$ in each iteration. Also the output result of the transition hypermatrix
is compatible with different values of $m$, which means that $m$ does not need to be pre-selected to run
the algorithm. In other words, we can stop the iteration anytime we want and still get the output
hypermatrix for some smaller $m$. This is useful when we later discuss how to choose the value of $m$, as it turns out that we can set a criterion to stop the iteration.
Lastly we see that the output transition hypermatrix $\cmP$ is only determined by $\mH$. So the procedure of computing $\cmP$ is similar to loading the matrix into the memory as they both only need to be done once for different problems (i.e., different $\vh$ and $\vb$). On the other hand, this means that we need global computation to compute this sequence and this choice prohibits a purely local algorithm.

\begin{theorem}
\label{the_mini}
Let $\cmP$ be the output of Algorithm~\ref{alg_1}, then $\tilde{\mH}$ defined in Theorem~\ref{the_var} has reached its minimal infinity norm. 
\end{theorem}

\begin{proof}


We use matrices $\mP^{(i)}$ for $i=1,2,\cdots,m$ to denote matrix slices of hypermatrices $\cmP$ from the output of Algorithm~\ref{alg_1}, and $\eta^{(k)}_i$ for the value of $\eta_i$ at the $k$th iteration.

We first prove that the value of $\normof{\tilde{\mH}}$
cannot be further decreased by changing $\mP^{(m)}$. Since
$\tilde{\mH}$ is a nonnegative matrix, $\normof{\tilde{\mH}} = \normof{\tilde{\mH}\ve}$ holds, where $\ve\in \mathbb{R}^n$  and $e_i=1$ for all $i=1,2,\cdots,n$. The $k$th element of $\hat{\mH}^{(m)}\ve$ is $\sum_{j=1}^{n}H^2_{k,j}/\cP^{(m)}_{k,j}$, and according to Lemma~\ref{lem_bound}, this value is minimized when $\cP^{(m)}_{k,i} = |H_{k,i}|/\sum_{j=1}^{n}|H_{k,j}|$, which is exactly the $k$th row of $\mP^{(m)}$ from the algorithm. Thus by changing $\mP^{(m)}$, we cannot decrease any elements of vector $\hat{\mH}^{(m)}\ve$, and $\normof{\tilde{\mH}\ve}$ will not decrease.

Second we prove that $\big( (\eta^{(m)}_1)^2, (\eta^{(m)}_2)^2,\cdots,(\eta^{(m)}_n)^2 \big)^T= \hat{\mH}^{(m)}\ve $. Because $\cmP^{(m)}$ is constructed as $\cP^{(m)}_{k,i} = |H_{k,i}|/\sum_{j=1}^{n}|H_{k,j}|$, which means the $k$th element of
$\hat{\mH}^{(m)}\ve$ is $(\sum_{j=1}^{n}|H_{k,j}|)^2 = (\eta^{(m)}_k)^2$.

Lastly we use mathematical induction and assume that we cannot decrease $\normof{\tilde{\mH}}$ by changing $\mP^{(\ell)}$, and $\big( (\eta^{(\ell)}_1)^2, (\eta^{(\ell)}_2)^2,\cdots,(\eta^{(\ell)}_n)^2 \big)^T= \hat{\mH}^{(\ell)}\cdots\hat{\mH}^{(m)}\ve $ for $ r+1 \leq \ell\leq m$. Then similarly we prove that the statement holds for $\mP^{(r)}$. 
We notice that the $k$th element of $\hat{\mH}^{(r)}\hat{\mH}^{(r+1)}\cdots\hat{\mH}^{(m)}\ve$ is
$\sum_{j=1}^{n} (H_{k,j}\eta^{(r+1)}_{j})^2/\cP^{r}_{k,j}$ and it is minimized because $\cmP^{(r)}$ is computed
as
\begin{equation}
\label{equ_prij}
\cP^{(r)}_{i,j} = \eta^{(r+1)}_{j}|H_{i,j}|/\sum_{k=1}^{n}\eta^{(r+1)}_{k}|H_{i,k}|.
\end{equation}
So no elements of the vector $\hat{\mH}^{(r)}\hat{\mH}^{(r+1)}\cdots\hat{\mH}^{(m)}\ve$ will decrease in value and neither will norm $\normof{\tilde{\mH}}$ if we change $\mP^{(r)}$. From
formula~\eqref{equ_prij} we can compute the $k$th element of $\hat{\mH}^{(r)}\hat{\mH}^{(r+1)}\cdots\hat{\mH}^{(m)}\ve$ as $(\sum_{j=1}^{n}\eta^{(r+1)}_{j}|H_{k,j}|)^2 = (\eta^{(r)}_{k})^2$.
So we have proved that this induction statement also holds for $\ell = r$. In conclusion the output hypermatrix $\cmP$ from Algorithm~\ref{alg_1} will ensure $\normof{\tilde{\mH}}$ to be minimized.
\end{proof}

The standard $1-$way method can also be viewed as a special case of $m-$way random walk, with the $m$ transition matrices being the same. However the $1-$way method generally does not minimize $\normof{\tilde{\mH}}$ in the $m-$way setting
as Algorithm~\ref{alg_1} minimizes $\normof{\tilde{\mH}}$. Formula
\eqref{equ_variance} from Theorem~\ref{the_var} indicates the connection between the variance and the power series of $\tilde{\mH}$. Since $\normof{\tilde{\mH}}$ is
an upper-bound of $\rho(\tilde{\mH})$ and $\rho(\tilde{\mH})$ affects how big this
power series will grow, we can see that the $m-$way random walk with transition hypermatrix defined from Algorithm~\ref{alg_1} has the tendency to decrease the variance compared to the standard $1-$way method. Although the above analysis does not ensure a smaller variance for the $m-$way
method, numerical experiments in both synthetic matrices and matrices in real applications support this conjecture. (See Section~\ref{sec_exp}).

Next we move to see how the spectral radius $\rho(\tilde{\mH})$ is related to the matrix $\mH$. In order to bound $\rho(\tilde{\mH})$, the standard Markov random walk
requires $\normof{\mH}<1$, which does not happen often from our early analysis. The following theorem
states the necessary and sufficient condition for a $m-$way Markov random walk to have
$\rho(\tilde{\mH}) < 1$.

\begin{theorem}
\label{the_ifonlyif}
Let $\mH^{+}$ denote the nonnegative matrix where $H^{+}_{i,j} = |H_{i,j}|$.
There exists a $m-$way Markov random walk transition hypermatrix $\cmP$ such $\normof{\tilde{\mH}} < 1$
if and only if $\rho(\mH^{+}) < 1$.
\end{theorem}

\begin{proof}
If there exists a $m-$way Markov random walk transition hypermatrix $\cmP$ such that $\normof{\tilde{\mH}} < 1$, without a loss of generality we assume $\cmP$ is the output from Algorithm~\ref{alg_1} since Theorem~\ref{the_mini} states that it minimize $\normof{\tilde{\mH}}$. From the proof of 
Theorem~\ref{the_mini} we have:
\begin{equation}
\label{equ_eta}
\normof{\tilde{\mH}} = \normof{\tilde{\mH}\ve} = 
\normof{\hat{\mH}^{(1)}\hat{\mH}^{(2)}\cdots\hat{\mH}^{(m)}\ve}
= \normof{\big( (\eta^{(1)}_1)^2, (\eta^{(1)}_2)^2,\cdots,(\eta^{(1)}_n)^2 \big)^T}.
\end{equation}
According to the computing procedure of Algorithm~\ref{alg_1} we have $( \eta^{(\ell)}_1, \eta^{(\ell)}_2,\cdots,\eta^{(\ell)}_n )^T = (\mH^{+})^m\ve$. So $\normof{\tilde{\mH}} < 1 \Longrightarrow \normof{(\mH^{+})^m\ve} < 1
\Longrightarrow \normof{(\mH^{+})^m} < 1
\Longrightarrow \rho(\mH^{+}) < 1$.

If we have $\rho(\mH^{+}) < 1$, from Gelfand's Formula, we have $\rho(\mH^{+})  = 
\lim_{k\rightarrow \infty}\normof{(\mH^{+})^k}^{1/k}$. Then we can find a sufficient large number $m$ s.t.
for any $k\geq m$ the inequality $\normof{(\mH^{+})^k}^{1/k} < 1$ holds. Let $\tilde{\mH}$
be the matrix based on the transition hypermatrix output from Algorithm~\ref{alg_1}. Based on
the observation of~\eqref{equ_eta}, we have 
$
\rho(\mH^{+}) < 1
\Longrightarrow \normof{(\mH^{+})^m}<1
\Longrightarrow \eta^{(1)}_{i} < 1, i=1,2,\cdots,n
\Longrightarrow \normof{\tilde{\mH}}<1
$
\end{proof} 

Theorem~\ref{the_ifonlyif} creates an equivalent link between $\rho(\mH^{+}) < 1$
and existence of $m-$way Markov random walk such that $\normof{\tilde{\mH}} < 1$. However it does not guarantee
the size of $m$. In another words one can always cook up some matrix
$\mH$ with $\rho(\mH^{+}) < 1$ but make $m$ arbitrarily large. Although these
extreme cases
are not our primary focus in this paper, we point it out for the discussion of the practical implementation of Algorithm~\ref{alg_1}. In order to find the transition hypermatrix $\cmP$
with $\normof{\tilde{\mH}} < 1$, we can set a threshold number $\phi_{\max}$, and let $m$ grow
until we have $\eta_i<1$ for all $i=1,2,\cdots,n$ or $m = \phi_{\max}$. As stated before, we do
not need to re-run the algorithm for different value of $m$, because the way Algorithm
\ref{alg_1} computes the transition hypermatrix is compatible with different values
of $m$.

\subsection{Random Walk Error Analysis}
To practically estimate the value $\langle \vh,\vx \rangle$ from simulating the value of $Z$,
we need to truncate the multi-way Markov random walk in order for it to end after some large
number of steps $N$. The practical solution~\cite{dimov1998new,benzi13analysis} to determine $N$ is through the 
criterion: $\lvert W_{N}\rvert \leq \epsilon \lvert W_0 \rvert$ where $\epsilon>0$ denotes some small number. For the case that the initial probability $p_i = |h_i|/
\sum_{j=1}^{n}|h_j|$, we have $W_0 = \normof{\vh}$.

We notice that $W_N$ is a random variable, and
follow the similar analysis with that in Theorem~\ref{the_exp}, it is easy to see its expected value is
$\langle \vh, (\mH^{+})^{N}\ve \rangle$. So $\rho(\mH^{+}) < 1$ is a necessary condition in 
order to determine the truncation number $N$. Here we can see that our $m-$way Markov random
walk has the minimal requirements on $\mH$, because $\rho(\mH^{+}) < 1$ is required for all the Monte Carlo frameworks to be able to truncate the random walk, and yet we show that under this condition, our algorithm can always find a $m-$way transition Hypermatrix
to ensure $\normof{\tilde{\mH}} < 1$.

The following theorem justifies that the truncation procedure has little effect on the estimation result or the variance of the variable.

\begin{theorem}
\label{the_trun}
Let $Z_N$ denote the truncation value of $Z$ after $N$ steps of the random walk. Formally
$Z_N = \sum_{\ell = 0}^{N} W_{\ell}b_{k_{\ell}}$ with $W_{\ell}, \ell = 0,1,2,\cdots,N$ defined in equation~\eqref{equ_weight}. 
If $\normof{\tilde{\mH}} < 1$ then $Z_N$
converges in probability to $Z$: $Z_N \stackrel{\mathclap{\small\normalfont\mbox{$p$}}}{\longrightarrow} Z$, and $\Var[Z_N]$ converges to $\Var[Z]$ as $N \rightarrow \infty$.
\end{theorem}

\begin{proof}
From the definition of the variable $Z = \lim_{N\rightarrow \infty} Z_N$, the conclusions can be easily verified.
\end{proof}

In addition to the truncation, another error comes from the simulation procedure when using
the empirical mean value of $Z$ to estimate $\E[Z]$, formally we define the probable error as:
\begin{equation*}
r = \sup\Big\{ s: \text{Pr}\big( |\bar{Z} - \E[Z] | \geq s\big)> \frac{1}{2} \Big\}
\end{equation*}
where $\bar{Z} = \sum_{i=1}^{M}Z^{(i)}/M$ denotes the mean value of $M$ simulations
$Z^{(1)},Z^{(2)},\cdots,Z^{(M)}$.

There is a close link between the probable error and the variance of the random variable. According 
to Central Limit Theorem
\begin{equation*}
\sqrt{M}\big(|\bar{Z} - Z|\big) \stackrel{\mathclap{\small\normalfont\mbox{$d$}}}{\longrightarrow} 
\mathcal{N}\big(0,\Var[Z]\big)
\end{equation*}
where $\mathcal{N}\big(0,\Var[Z]\big)$ denotes the normal distribution with zero mean and variance $\Var[Z]$, and the symbol $\stackrel{\mathclap{\small\normalfont\mbox{$d$}}}{\longrightarrow} $ means convergence in distribution. When $M$ is sufficiently large, $r \approx 0.6745 \sqrt{\Var[Z]/M}$.

The probable error is determined by the ratio of the variance to the number of simulations. If the variance is decreased by $\xi$ times, then it only require $\xi$ times fewer number of simulations
to get to the same precision (i.e., probable error). 

Based on the above observations we can conduct
numerical experiments to compare the variance between the standard Monte Carlo method and our
multi-way Monte Carlo method, and the ratio between the variance can demonstrate how many times
faster our new method can get.

\section{Numerical Experiments}
\label{sec_exp}

In this section, we conduct numerical experiments to demonstrate the two key improvements from our new method~\footnote{Codes for this paper are available at \url{https://github.com/wutao27/multi-way-MC}}. In section \ref{sec_compare1} we show that our
new method can be applied solve more problems than the standard method.
In section \ref{sec_compare2} we show that our new method can achieve a considerable
speed-up. The testing methods are the standard $1-$way method, and our multi-way methods with $m=2,3,4,5$.
In both experiments, synthetic matrices and real-world matrices are used.

For synthetic data, we generate the matrix $\mH$ via Matlab command $\texttt{sprand}(1000,1000,0.2)$, which outputs
a $1,000$ by $1,000$ matrix with $0.2$ of its entries being non-zeros, and each non-zero is a random number following uniform distribution between $(0,1)$. The synthetic matrices are rescaled to reach certain spectral radius required during experiments. Formally to get a spectral radius $0<r<1$: $\mH \leftarrow r\mH/\rho(\mH)$. Each result is the average over 100 trials for the related problems.

For real world matrices, we focus on the Harwell-Boeing sparse matrix collection~\cite{duff1992users,davis2011university}. The matrix $\mH$ is constructed by a simple diagonal precondition on the original matrix $\mA$ from the collection: $\mH = \mI - \Diag(\mA)^{-1}\mA$. And for the test problems we only consider the matrices that have $\rho(\mH^{+}) < 1$. In the interest of simplicity, we only use problems with fewer than 5,000 dimensions.

In both synthetic and real world experiments, vectors $\vb,\vh$ are randomly generated with elements following uniform distribution between $(0,1)$.

\subsection{The Number of Solvable Problems}
\label{sec_compare1}
We define the solvable problems as those with $\normof{\tilde{\mH}} < 1$, which is a sufficient condition that guarantees the convergence of the Monte Carlo methods. The ratio of solvable problems is the percentage of random problems that are solvable. Figure \ref{fig_ratio} shows the results as we vary the spectral radius. As we can see our multi-way methods can solve more problems than the standard method, which cannot guarantee any convergences when $\rho(\mH^{+}) \geq 0.85$. And when $m$ increases, even more problems  are solvable. We also find several real world matrices where the standard
method fails to guarantee convergence but ours can. They are matrices \textit{fs\_760\_1, jpwh\_991, nos7} from the Harwell-Boeing Collection and \textit{add32} from Hamm matrix group\footnote{\url{http://www.cise.ufl.edu/research/sparse/matrices/Hamm/index.html}}.

\begin{figure}[htb]
\begin{floatrow}
\ffigbox{%
  \includegraphics[trim={2.3cm 7.5cm 2.3cm 7.5cm}, width=0.82\linewidth]{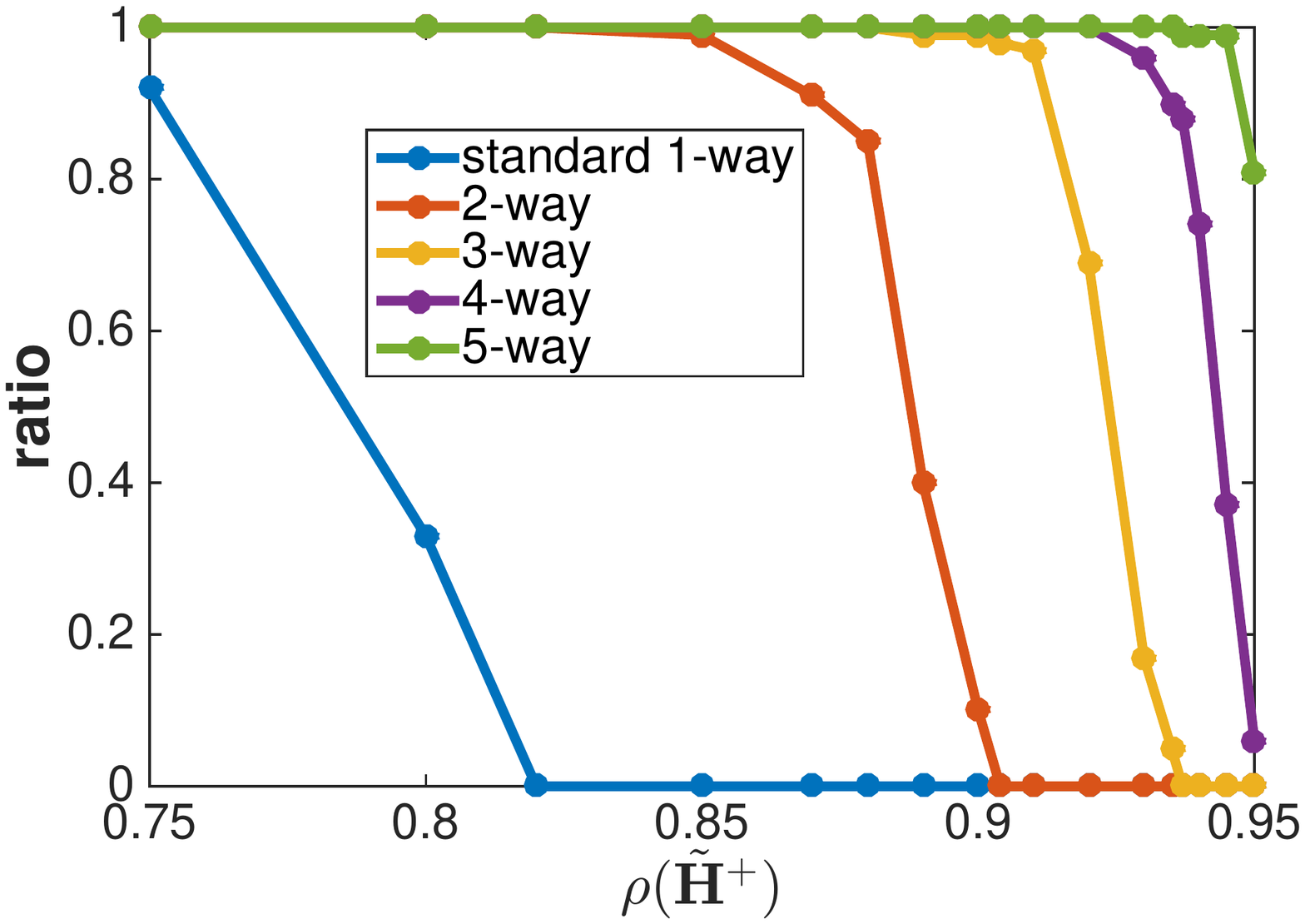}%
}{%
  \caption{The results of the standard $1-$way method and our multi-way methods with $m=2,3,4,5$ for the ratio of solvable synthetic problems vs the spectral radius $\rho(\tilde{\mH^+})$.
  }\label{fig_ratio}%
}
\capbtabbox{%
  \begin{tabularx}{0.98\linewidth}{ >{\hsize=1.12\hsize}X >{\hsize=0.95\hsize}Z >{\hsize=0.95\hsize}Z >{\hsize=0.95\hsize}Z>{\hsize=0.9\hsize}Z}
\toprule
 & $2-$way & $3-$way & $4-$way & $5-$way  \\
\cmidrule{2-5}
 $\rho(\mH^+)$ &\multicolumn{4} {>{\hsize=4\hsize}Z}{Synthetic Matrices} \\
\midrule
$0.8$  & $1.09$ & $1.13$ & $1.14$ & $1.15$ \\
$0.9$ & $1.38$ & $1.58$ & $1.69$ & $1.77$\\
$0.95$ & $1.75$ & $2.30$ & $2.73$ & $3.06$\\
$0.99$ & $2.40$ & $3.77$ & $5.10$ & $6.39$\\
\midrule
&\multicolumn{4} {>{\hsize=4\hsize}Z}{Harwell-Boeing Collection} \\
\midrule
 & $1.20$ & $1.44$ & $1.59$ & $1.74$\\
\bottomrule
\end{tabularx}
}{%
  \caption{The speed-up times by our multi-way methods with $m=2,3,4,5$
  compared to the standard $1-$way method on synthetic problems and the Harwell-Boeing collection. $r$ denotes $\rho(\mH^{+})$.}\label{table_ratio}
}
\end{floatrow}
\end{figure}

\subsection{Algorithm Efficiency}
\label{sec_compare2}
We apply the conclusion in Theorem \ref{equ_variance} to compute $\Var[Z]$ for all the testing methods. Then we compare $\Var[Z]$ for different methods. Formally we define speed-up times as $\Var[X]/\Var[Z]$, where
$X$ and $Z$ denote the variable from the standard $1-$way method and our method respectively. The speed-up times is an indicator for how
much times faster our multi-way methods can get compared to the standard $1-$way method.
Table \ref{table_ratio} shows that we have considerable speed-up when applying our multi-way methods.
Note that we only consider the problems with $\rho(\tilde{\mH}) < 1$ in order for $\Var[Z] < \infty$. In the Harwell-Boeing collection there are a few problems having $\mH^{+}\ve$ equal for each element, we exclude these matrices because the multi-way method is equivalent to the standard method, as we now show. 

When $\mH^{+}\ve$ is a vector with its elements being the same number, that is, $\mH^{+} \ve = \gamma \ve$, the multi-way method is equivalent to the standard method. This occurs because, as in 
Algorithm~\ref{alg_1}, the vectors $\bm{\eta} = (\mH^{+})^{k}\ve$ and $\bm{\omega} = (\mH^{+})^{k-1}\ve$ will have their elements be the same. So the transition matrices $\cmP^{(i)}$ for $i=1,2,\cdots,m$ will also be the same.
In our experiments on the Harwell-Boeing collection, we do find several matrices that have this property, so in this case our multi-way method is equivalent to the standard method. 

Among our testing problems, there is only one~\footnote{\url{http://www.cise.ufl.edu/research/sparse/matrices/HB/fs_760_1.html}} that our multi-way method can have a larger variance than that of the standard method. Actually we find the matrix $\mH$ for this problem is outside our assumptions in this paper. We assume that $\mH$ does not have zero row in order to assign transition probabilities for each state. For the corner case that $\mH$ does have zero
rows, the linear system can be easily adjusted by deleting the zero rows of $\mH$. For this testing problem $\mH$, the row sums of $\mH^{+}$ distribute in a drastic way. Over half of the rows have sum values between $10^{-17}$ to $10^{-6}$, and quite a few ``big'' rows have sums larger than $10^{3}$. So this matrix have many rows that are nearly zero. For all the other testing problems, our multi-way method can achieve smaller variances than the standard method, and the speed-up times in shown in Table~\ref{table_ratio}.


\section{Conclusion}
\label{sec_con}
In this paper we studied a generalization of Monte Carlo methods for linear systems.
The generalization allows the Markov random walk to transition using a set of matrices. We derived the variance of the resulting estimator and construct the matrices in a way to attempt to produce a finite variance.  The advantages of this new random walk procedures are two-fold. First it can solve more problems that the standard method fails to solve.
Second our new method has the tendency to decrease the variance
thus decrease the computations needed for estimate the solution.
Numerical experiments on both synthetic and real world matrices confirm the
superiority of our method in the above two aspects when comparing to the standard Monte Carlo method. An open problem suggested by our work is to get a purely local method that avoids the global work in building the sequence of adjacency matrices.


\textbf{Acknowledgements.} This work was supported by NSF IIS-1422918,
CAREER award CCF-1149756, Center for Science of Information STC, CCF-093937; DOE award DE-SC0014543; and the DARPA SIMPLEX program.

\bibliographystyle{abbrv}
\bibliography{refs_mc}

\begin{thebibliography}{10}

\bibitem{alexandrov2005parallel}
V.~Alexandrov, E.~Atanassov, I.~Dimov, S.~Branford, A.~Thandavan, and
  C.~Weihrauch.
\newblock Parallel hybrid {Monte Carlo} algorithms for matrix computations.
\newblock In {\em International Conference on Computational Science}, pages
  752--759. Springer, 2005.

\bibitem{avrachenkov2007monte}
K.~Avrachenkov, N.~Litvak, D.~Nemirovsky, and N.~Osipova.
\newblock {Monte Carlo} methods in {PageRank} computation: When one iteration
  is sufficient.
\newblock {\em SIAM Journal on Numerical Analysis}, 45(2):890--904, 2007.

\bibitem{benzi13analysis}
M.~Benzi, T.~Evans, S.~Hamilton, M.~L. Pasini, and S.~Slattery.
\newblock Analysis of {Monte Carlo} accelerated iterative methods for sparse
  linear systems.
\newblock Technical Report Math/CS Technical Report TR-2016-002, Emory
  University, 2016.

\bibitem{davis2011university}
T.~A. Davis and Y.~Hu.
\newblock The university of florida sparse matrix collection.
\newblock {\em ACM Transactions on Mathematical Software (TOMS)}, 38(1):1,
  2011.

\bibitem{dietrich1996scalar}
S.~Dietrich and I.~D. Boyd.
\newblock Scalar and parallel optimized implementation of the direct simulation
  {Monte Carlo} method.
\newblock {\em Journal of Computational Physics}, 126(2):328--342, 1996.

\bibitem{dimov2001parallel}
I.~Dimov, V.~Alexandrov, and A.~Karaivanova.
\newblock Parallel resolvent {Monte Carlo} algorithms for linear algebra
  problems.
\newblock {\em Mathematics and Computers in Simulation}, 55(1):25--35, 2001.

\bibitem{dimov1998new}
I.~Dimov, T.~Dimov, and T.~Gurov.
\newblock A new iterative {Monte Carlo} approach for inverse matrix problem.
\newblock {\em Journal of Computational and Applied Mathematics}, 92(1):15--35,
  1998.

\bibitem{dimov2015new}
I.~Dimov, S.~Maire, and J.~M. Sellier.
\newblock A new walk on equations {Monte Carlo} method for solving systems of
  linear algebraic equations.
\newblock {\em Applied Mathematical Modelling}, 39(15):4494--4510, 2015.

\bibitem{duff1992users}
I.~S. Duff, R.~G. Grimes, and J.~G. Lewis.
\newblock Users' guide for the harwell-boeing sparse matrix collection (release
  i), 1992.

\bibitem{evans2014monte}
T.~M. Evans, S.~W. Mosher, S.~R. Slattery, and S.~P. Hamilton.
\newblock A {Monte Carlo} synthetic-acceleration method for solving the thermal
  radiation diffusion equation.
\newblock {\em Journal of Computational Physics}, 258:338--358, 2014.

\bibitem{halton1994sequential}
J.~H. Halton.
\newblock Sequential {Monte Carlo} techniques for the solution of linear
  systems.
\newblock {\em Journal of Scientific Computing}, 9(2):213--257, 1994.

\bibitem{ji2013convergence}
H.~Ji, M.~Mascagni, and Y.~Li.
\newblock Convergence analysis of {Markov} chain {Monte Carlo} linear solvers
  using {Ulam-von Neumann} algorithm.
\newblock {\em SIAM Journal on Numerical Analysis}, 51(4):2107--2122, 2013.

\bibitem{lebeau1999parallel}
G.~LeBeau.
\newblock A parallel implementation of the direct simulation {Monte Carlo}
  method.
\newblock {\em Computer Methods in Applied Mechanics and Engineering},
  174(3):319--337, 1999.

\bibitem{slattery2013parallel}
S.~R. Slattery.
\newblock {\em Parallel {Monte Carlo} Synthetic Acceleration methods for
  discrete transport problems}.
\newblock PhD thesis, University of Wisconsin Madison, 2013.

\bibitem{slattery2015spectral}
S.~R. Slattery, T.~M. Evans, and P.~P. Wilson.
\newblock A spectral analysis of the domain decomposed {Monte Carlo} method for
  linear systems.
\newblock {\em Nuclear Engineering and Design}, 295:632--638, 2015.

\bibitem{srinivasan2010monte}
A.~Srinivasan.
\newblock {Monte Carlo} linear solvers with non-diagonal splitting.
\newblock {\em Mathematics and Computers in Simulation}, 80(6):1133--1143,
  2010.

\bibitem{wang2008monte}
Q.~Wang, D.~Gleich, A.~Saberi, N.~Etemadi, and P.~Moin.
\newblock A {Monte Carlo} method for solving unsteady adjoint equations.
\newblock {\em Journal of Computational Physics}, 227(12):6184--6205, 2008.

\bibitem{wasow1952note}
W.~Wasow.
\newblock A note on the inversion of matrices by random walks.
\newblock {\em Mathematical Tables and Other Aids to Computation},
  6(38):78--81, 1952.

\end{thebibliography}

\end{document}